\documentclass[12pt]{article}
\usepackage{amsmath}
  \usepackage{paralist}
  \usepackage{amssymb}
  \usepackage{amsthm}
  \usepackage{graphics} %% add this and next lines if pictures should be in esp format
  \usepackage{epsfig} %For pictures: screened artwork should be set up with an 85 or 100 line screen
\usepackage{graphicx}  \usepackage{epstopdf}%This is to transfer .eps figure to .pdf figure; please compile your paper using PDFLeTex or PDFTeXify.
 \usepackage[colorlinks=true]{hyperref}
\hypersetup{urlcolor=blue, citecolor=red}
\newtheorem{theorem}{Theorem}[section]
\newtheorem{corollary}{Corollary}
\newtheorem{lemma}[theorem]{Lemma}
\newtheorem{proposition}{Proposition}

\newtheorem{remark}{Remark}

\begin{document}
%% Place the running title of the paper with 40 letters or less in []
 %% and the full title of the paper in { }.
%\title[The Camassa-Holm Equation] %Use the shortened version of the full title
 \begin{center}    
      THE CAMASSA-HOLM EQUATION AS THE LONG-WAVE LIMIT OF THE IMPROVED BOUSSINESQ EQUATION AND OF A CLASS OF  NONLOCAL WAVE EQUATIONS
 \end{center}

% Place all authors' names in [ ] shown as running head, Leave { } empty
% Please use `and' to connect the last two names if applicable
% Use FirstNameInitial.  MiddleNameInitial. LastName, or only last names of authors if there are too many authors
%\author[H.A. Erbay, S. Erbay and A. Erkip]{}

% It is required to enter 2010 MSC.
%\subjclass{Primary: 35Q53, 35Q74;        Secondary:  74J30, 35C20.}
% Please provide minimum  5 keywords.
%\keywords{Camassa-Holm equation, improved Boussinesq equation, nonlocal wave equation, rigorous justification.}

% Email address of each of all authors is required.
% You may list email addresses of all other authors, separately.
% \email{husnuata.erbay@ozyegin.edu.tr}
% \email{saadet.erbay@ozyegin.edu.tr}
% \email{albert@sabanciuniv.edu}

% Put your short thanks below. For long thanks/acknowlegements,
%please go to the last acknowlegments section.
%\thanks{The first author is supported by NSF grant xx-xxxx}

% Add corresponding author at the footnote of the first page if it is necessary.
% Plase add $^*$ adjacent to the corresponding author's name on the first page.
% The example shown in this template is if the first author is the corresponding author.
%\thanks{$^*$ Corresponding author: H.A. Erbay}

%\begin{document}
%\maketitle

% Enter the first author's name and address:
\centerline{\scshape H.A. Erbay$^*$ and S. Erbay}
\medskip
{\footnotesize
% please put the address of the first author
 \centerline{ Department of Natural and Mathematical Sciences, Faculty of Engineering,}
   \centerline{Ozyegin University,}
   \centerline{ Cekmekoy 34794, Istanbul, Turkey}
} % Do not forget to end the {\footnotesize by the sign }

\medskip

\centerline{\scshape A. Erkip}
\medskip
{\footnotesize
 % please put the address of the second  and third author
 \centerline{ Faculty of Engineering and Natural Sciences,}
   \centerline{Sabanci University,}
   \centerline{Tuzla 34956,  Istanbul,    Turkey}

}

\bigskip

% The name of the associate editor will be entered by an editorial staff
% "Communicated by the associate editor name" is not needed for special issue.
% \centerline{(Communicated by the associate editor name)}

%The abstract of your paper
\begin{abstract}
In the present study we prove rigorously that in the long-wave limit, the unidirectional solutions of  a class of nonlocal wave equations to which the improved Boussinesq equation belongs  are well approximated by the solutions of the  Camassa-Holm equation over a long time scale. This general class of nonlocal wave equations  model bidirectional wave propagation in a nonlocally and nonlinearly elastic medium whose constitutive equation is given by a convolution integral. To justify the Camassa-Holm approximation  we show that approximation errors   remain small over a long time interval. To be more precise, we obtain error estimates in terms of two independent, small, positive parameters $\epsilon$ and $\delta$ measuring the effect of nonlinearity and dispersion, respectively.  We further show that similar conclusions are also valid for the lower order approximations: the  Benjamin-Bona-Mahony approximation and the Korteweg-de Vries approximation.
\end{abstract}

\section{Introduction}\label{sec1}

In the present paper we rigorously prove that, in the long-wave limit and on a relevant time interval,  the right-going solutions of both the improved Boussinesq (IB) equation
\begin{equation}
    u_{tt}-u_{xx}-\delta^{2}u_{xxtt}-\epsilon (u^{2})_{xx}=0,  \label{ib}
\end{equation}%
and, more generally, the nonlocal wave equation
\begin{equation}
    u_{tt}=\beta_{\delta }\ast ( u+\epsilon u^{2})_{xx} \label{nw}
\end{equation}%
are well approximated by the solutions of the Camassa-Holm (CH) equation
\begin{equation}
    w_t+ w_x+ \epsilon w w_x -\frac{3}{4}\delta^{2} w_{xxx}-{\frac{5}{4}}\delta^{2}w_{xxt} -{\frac{3}{4}}\epsilon \delta^{2}(2 w_x w_{xx}+ w w_{xxx})=0.
    \label{ch}
\end{equation}
In the above equations, $u=u(x,t)$ and $w=w(x,t)$ are  real-valued functions, $\epsilon $ and $\delta $ are two small positive parameters measuring the effect of nonlinearity and dispersion, respectively, the symbol $\ast$ denotes convolution in the $x$-variable, $\beta_{\delta }(x) =\frac{1}{\delta }\beta (\frac{x}{\delta })$ is the kernel function.  It should be noted that  (\ref{ch}) can be written in a more standard form by means of a coordinate transformation. That is,   in a moving frame defined by $\bar{x}={\frac{2}{\sqrt 5}}(x-{\frac{3}{5}}t)$ and $\bar{t}={\frac{2}{3 \sqrt 5}}t $,  (\ref{ch})  becomes
\begin{equation}
    v_{\bar{t}}+ {\frac{6}{5}} v_{\bar{x}}+ 3\epsilon v v_{\bar{x}} -\delta^2 v_{\bar{t}\bar{x}\bar{x}}
        -{\frac{9}{5}}\epsilon \delta^2(2 v_{\bar{x}} v_{\bar{x}\bar{x}}+ v v_{\bar{x}\bar{x}\bar{x}})=0,  \label{chv}
\end{equation}
with $v(\bar{x},\bar{t})=w(x,t)$. Also,  by the use of the scaling transformation $U(X, \tau) =\epsilon u(x, t), ~~x=\delta X,  ~~t=\delta \tau$, (\ref{ib}) and (\ref{ch}) can be written in a more standard form with no parameters, but the above forms of (\ref{ib}) and (\ref{ch}) are more suitable to deal with small-but-finite amplitude long wave solutions.

In the literature, there have been a number of works concerning rigorous justification of the model equations derived for the unidirectional propagation of long waves from nonlinear wave equations modeling various physical systems. One of these model equations is the CH equation \cite{camassa,ionescu,johnson1} derived for the unidirectional propagation of long water waves in the context of a shallow water approximation to the Euler equations of inviscid incompressible fluid flow. The CH equation has attracted much attention  from researchers  over the years. The two main properties of the CH equation are: it is an infinite-dimensional completely integrable Hamiltonian system and it captures wave-breaking of water waves (see \cite{constantin1,constantin2, constantin3,lannes} for details). A rigorous justification of the CH equation for shallow water waves was given in \cite{constantin3}.

In a recent study \cite{eee1}, the CH equation has been also derived as an appropriate model for the unidirectional propagation of  long elastic waves  in an infinite,  nonlocally and nonlinearly elastic medium (see also \cite{eee2}).  The constitutive behavior of the nonlocally and nonlinearly elastic medium is described by a convolution integral (we refer the reader to \cite{duruk1, duruk2} for a detailed description of the nonlocally and nonlinearly  elastic medium) and in the case of quadratic nonlinearity the one-dimensional equation of motion  reduces to the nonlocal equation given in (\ref{nw}). Moreover, the nonlocal equation (that is,  the equation of motion for the  medium) reduces to the IB equation (\ref{ib}) for a particular choice of the kernel function appearing in the integral-type constitutive relation (see Section \ref{sec5} for details). In order to derive formally the CH equation from the IB equation, an asymptotic expansion valid  as nonlinearity and dispersion parameters, that is $\epsilon $ and $\delta $, tend to zero independently is used  in \cite{eee1}. It has been also pointed out  that a similar formal derivation of the CH equation  is possible by starting from the nonlocal equation (\ref{nw}).

The question that naturally arises is under which conditions the unidirectional solutions of  the nonlocal  equation are well approximated by the solutions of the CH equation and this is the subject of the present study. Given a solution of the CH equation we find the corresponding solution of the nonlocal equation and show that the approximation error, i.e. the difference between the two solutions, remains small in suitable norms on a relevant time interval.  We conclude that  the CH equation is an appropriate  model equation for the unidirectional propagation of nonlinear dispersive elastic waves. The methodology used in this study adapts the techniques in \cite{bona,constantin3,gallay}.

We note that, in the terminology of some authors,  our results are in fact consistency-existence-convergence results for the CH approximation of the IB equation and, more generally, of the nonlocal equation. We refer to \cite{bona} and the references therein for a detailed discussion of these concepts.

As it is pointed above, the general class of nonlocal wave equations contains the IB equation as a member. Therefore, to simplify our presentation,  we start with the CH approximation of the IB equation and then extend the analysis to the case of the general class of  nonlocal wave equations. Though our analysis is mainly concerned with the CH approximations of the IB equation and the nonlocal equation, our results apply as well to the Benjamin-Bona-Mahony  (BBM) approximation. We also show how to use our results to justify the Korteweg-de Vries (KdV) approximation.

The structure of  the paper is as follows. In Section \ref{sec2} we observe that the solutions of the CH equation  are uniformly bounded in suitable norms for all values of $\epsilon $ and $\delta $. In Section \ref{sec3} we estimate the residual term that arises when we plug the solution of the CH equation into the IB equation. In Section \ref{sec4}, using the energy estimate based on certain commutator estimates, we complete the proof of the main theorem. In Section \ref{sec5} we extend our consideration from the IB equation to the nonlocal equation and we prove a similar theorem for the nonlocal equation. Finally, in Section \ref{sec6} we give error estimates for the long-wave approximations based on the BBM  equation  \cite{bbm}  and the  KdV  equation  \cite{korteweg}.

Throughout this paper, we use the standard notation  for  function spaces. The Fourier transform of $u$, defined by $\widehat u(\xi)=\int_\mathbb{R} u(x) e^{-i\xi x}dx$,  is denoted by the symbol $\widehat u$. The symbol $\Vert u\Vert_{L^p}$ represents the $L^p$ ($1\leq p <\infty$)  norm of $u$ on $\mathbb{R}$. The symbol $\langle u, v\rangle$ represents the inner product of $u$ and $v$ in $L^2$.   The notation  $H^{s}=H^s(\mathbb{R})$ denotes the $L^{2}$-based Sobolev space of order $s$ on $\mathbb{R}$,  with the norm $\Vert u\Vert_{H^{s}}=\left(\int_\mathbb{R} (1+\xi^2)^s |\widehat u(\xi)|^2 d\xi \right)^{1/2}$.  The symbol $\mathbb{R}$ in $\int_{\mathbb{R}}$ will be suppressed. $C$ is a generic positive constant.  Partial differentiations are denoted by $D_{t}$, $D_{x}$ etc.

\section{Uniform Estimates for the Solutions of the Camassa-Holm Equation}\label{sec2}

In this section, we observe that the solutions $w^{\epsilon, \delta }$ of the CH equation are uniformly bounded in suitable norms for all values of $\epsilon $ and $\delta $.  This  is a direct consequence of the estimates proved by Constantin and Lannes in \cite{constantin3} for a more general class of equations, containing the CH equation as a special case.

For convenience of the reader,  we rephrase below Proposition 4 of \cite{constantin3}.  To that end, we first recall some definitions from \cite{constantin3}: {\it (i)} For every $s\geq 0$, the symbol $X^{s+1}(\mathbb{R})$ represents the space $ H^{s+1}\left( \mathbb{R}\right) $ endowed with the norm $\left\vert f\right\vert _{X^{s+1}}^{2}=\Vert f\Vert _{H^{s}}^{2}+\delta^{2} \Vert f_{x}\Vert _{H^{s}}^{2},$ and {\it (ii )} the symbol $\mathcal{P}$ denotes the index set
\begin{equation*}
    \mathcal{P}=\left\{ (\epsilon ,\delta ) : 0<\delta <\delta_{0}, ~\epsilon \leq M\delta \right\}
\end{equation*}
for some $\delta_{0}>0$ and $M>0$. Then, Proposition 4 of \cite{constantin3} is as follows:
\begin{proposition}\label{prop2.1}
    Assume that \ $\kappa_{5} <0$ and let $\delta_{0}>0$, $M>0$, $s>\frac{3}{2},$ and $w_{0}\in H^{s+1}\left( \mathbb{R}\right)$. Then there exist $T>0$ and a unique family of solutions $\left\{ w^{\epsilon ,\delta}\right\}_{(\epsilon, \delta) \in \mathcal{P}}$ to the Cauchy problem
    \begin{align}
        & w_{t}+w_{x}+\kappa_{1}\epsilon ww_{x}+\kappa_{2}\epsilon^{2}w^{2}w_{x}+\kappa_{3}\epsilon^{3} w^{3}w_{x} +\delta^{2} \left( \kappa_{4} w_{xxx}+\kappa_{5} w_{xxt}\right) \nonumber \\
        &  ~~~~  -\epsilon \delta^{2} (\kappa_{6} ww_{xxx}+\kappa_{7} w_{x}w_{xx})=0, \label{cons-a} \\
        & w(x,0)=w_{0}(x)  \label{cons}
    \end{align}%
    (with constants $\kappa_{i} ~(i=1,2,...,7)$) bounded in $C\left( [0,\frac{T}{\epsilon }],X^{s+1}(\mathbb{R})\right) \cap C^{1}\left( [0,\frac{T}{\epsilon }],X^{s}(\mathbb{R})\right) $.
\end{proposition}
We refer the reader to  \cite{constantin3} for the proof of this proposition. Furthermore,   $T$ of the existence time $T/\epsilon$ is expressed  in \cite{constantin3} as
\begin{equation*}
    T=T\left(\delta_{0},M,\left\vert w_{0}\right\vert_{X_{\delta_{0}}^{s+1}},\frac{1}{\kappa_{5} },\kappa_{2},\kappa_{3}, \kappa_{6}, \kappa_{7}  \right)>0.~~~
\end{equation*}
Obviously, the CH equation (\ref{ch}) is a special case of (\ref{cons-a}) where $\kappa_{1}=1$, $\kappa_{2}=\kappa_{3}=0$, $\kappa_{4}=-\frac{3}{4} $, $\kappa_{5}=- \frac{5}{4}$ and $2\kappa_{6}=\kappa_{7}=-\frac{3}{2} $. In subsequent sections we will need to use uniform estimates for the terms  $\left\Vert w^{\epsilon ,\delta }\left( t\right) \right\Vert _{H^{s+k}}$  and $\left\Vert w_{t}^{\epsilon ,\delta }\left( t\right) \right\Vert _{H^{s+k-1}}$ with some $k\geq 1$. Proposition \ref{prop2.1} provides us with such estimates, nevertheless to avoid the extra $\delta^{2}$ term  in the $X^{s+1}$-norm, we will use a weaker version based on the inclusion $X^{s+k+1}\subset $ $H^{s+k}$.  Furthermore, for simplicity, we take   $\delta_{0}=M=1$. We thus reach the following corollary:
\begin{corollary}\label{cor2.1}
         Let $w_{0} \in H^{s+k+1}\left( \mathbb{R}\right) $,  $ s>1/2$, $k\geq 1$. Then, there exist $T>0$,  $C>0$  and a unique family of solutions
        \begin{equation*}
            w^{\epsilon ,\delta }\in C\left( [0,\frac{T}{\epsilon }],H^{s+k}(\mathbb{R})\right) \cap C^{1}\left( [0,\frac{T}{\epsilon }],H^{s+k-1}(\mathbb{R})\right)
        \end{equation*}%
        to the CH equation (\ref{ch}) with initial value  $w(x, 0)=w_{0}(x)$, satisfying
        \begin{equation*}
            \left\Vert w^{\epsilon ,\delta }\left( t\right) \right\Vert _{H^{s+k}}+\left\Vert w_{t}^{\epsilon ,\delta }(t) \right\Vert _{H^{s+k-1}}\leq C,
        \end{equation*}%
        for all $0<\delta \leq 1$, \ $\epsilon \leq \delta $ and $t\in \lbrack 0,\frac{T}{\epsilon }]$.
\end{corollary}

\section{Estimates  for the Residual Term Corresponding to the Camassa-Holm Approximation}\label{sec3}

Let $w^{\epsilon ,\delta }$ be the family of solutions mentioned in Corollary \ref{cor2.1} for the Cauchy problem of the CH equation with initial value  $w_{0} \in H^{s+k+1}\left( \mathbb{R}\right) $. In this section we estimate the residual term that arises when we plug $w^{\epsilon ,\delta }$ into the IB equation. Obviously, the residual term $f$  for the IB equation is
\begin{equation}
    f=w_{tt}-w_{xx}-\delta ^{2}w_{xxtt}-\epsilon (w^{2})_{xx},  \label{residual0}
\end{equation}%
where and hereafter  we drop the indices $\epsilon$, $\delta$ in $u$ and $w$ for simplicity.

Using the CH equation we now show that the residual term $f$ has a potential function. We start by rewriting  the CH equation in the form
\begin{equation}
    w_{t}+w_{x}=-\epsilon ww_{x}+{\frac{3}{4}} \delta^{2} w_{xxx}+{\frac{5}{4}} \delta^{2}  w_{xxt}
                +{\frac{3}{4}} \epsilon \delta^{2}  (2w_{x}w_{xx}+ww_{xxx}). \label{ch-right}
\end{equation}%
Using repeatedly (\ref{ch-right})  in (\ref{residual0}) we get
\begin{align}
    f =&\left( {D}_{t}-{D}_{x}\right) \left[ -\epsilon ww_{x}+{\frac{3}{4}}\delta^{2}w_{xxx}+{\frac{5}{4}}\delta^{2}w_{xxt}
            +{\frac{3}{4}}\epsilon \delta^{2}D_{x}(\frac{1}{2}w_{x}^{2}+ww_{xx})   \right]  \nonumber \\
    & ~~ ~~      -\delta^{2}w_{xxtt}-\epsilon (w^{2})_{xx}  \nonumber \\
    =& \epsilon^{2}D_{x}^{2}({\frac{w^{3}}{3}})-{\frac{3}{8}}\epsilon^2 \delta^2\left[ D_{x}^{2}(w_{x}^{2}+2ww_{xx})\right] \nonumber  \\
    &  ~~~~ +{\frac{1}{16}}\delta^4\left[ (D_{x}^{2}D_{t}-3D_{x}^{3})(3w_{xxx}+5w_{xxt})\right]  \nonumber  \\
    & ~~~~  +{\frac{3}{32}}\epsilon \delta^{4}\left[ (D_{x}^{3}D_{t}-3D_{x}^{4})(w_{x}^{2}+2ww_{xx}) \right]  \nonumber \\
    & ~~~~ + {\frac{1}{4}} \epsilon \delta^2 D_{x} \left[ (-3wD_{x}^{2}+2w_{xx}+w_{x}D_{x})(w_{t}+w_{x})\right]. \label{fff}
\end{align}
After some straightforward calculations we write  $f=F_{x}$  with
\begin{align*}
F =& \epsilon^2({\frac{w^{3}}{3}})_{x}-{\frac{1}{8}}\epsilon^2 \delta^2\left[ 3(w_{x}^{2}+2ww_{xx})_{x}
         -3w(w^{2})_{xxx}+2w_{xx}(w^{2})_{x}+w_{x}(w^{2})_{xx}\right]  \\
    & + {\frac{1}{16}} \delta^4\left[ (D_{x}D_{t}-3D_{x}^{2})(3w_{xxx}+5w_{xxt})\right]  \\
    &   +{\frac{1}{32}}\epsilon \delta^{4}\left[ 3(D_{x}^{2}D_{t}-3D_{x}^{3})(w_{x}^{2}+2ww_{xx}) \right.  \\
    & \left. +2(-3wD_{x}^{2}+2w_{xx}+w_{x}D_{x})(3w_{xxx}+5w_{xxt})\right] \\
    & +{\frac{1}{32}}\epsilon^2 \delta^4 \left[ (-9wD_{x}^{3}+6w_{xx}D_{x}+3w_{x}D_{x}^{2})(w_{x}^{2}+2ww_{xx})\right].
\end{align*}%
Note that, except for the term $D_{x}^{3}D_{t}^{2}w$,  $F$ is a combination of terms of the form $D_{x}^{j}w$ with  $j\leq 5$ or $D_{x}^{l}D_{t}w$ with  $l\leq 4$. By taking $k=5$ it immediately follows from Corollary \ref{cor2.1}  that all of the terms in $F$, except $D_{x}^{3}D_{t}^{2}w$, are uniformly bounded in the $H^{s}$ norm. To deal with the term $D_{x}^{3}D_{t}^{2}w$, we first rewrite the CH equation in the form
\begin{equation}
    w_{t}=\mathcal{Q}\left[-w_{x}-\epsilon ww_{x}+{\frac{3}{4}}\delta^{2}w_{xxx}
            +{\frac{3}{4}}\epsilon \delta^{2}(2w_{x}w_{xx}+ww_{xxx})\right], \label{wQop}
\end{equation}
where the operator $\mathcal{Q}$ is
\begin{equation}
   \mathcal{Q}=\left( 1-\frac{5}{4}\delta^{2}D_{x}^{2}\right)^{-1}.  \label{qoperator}
\end{equation}
Then, applying the operator $D_{x}^{3}D_{t}$ to (\ref{wQop}) and using (\ref{ch-right}) we get
\begin{align*}
    D_{x}^{3}D_{t}w_{t}
    =& D_{x}^{3}D_{t}\mathcal{Q}\left[-w_{x}-\epsilon ww_{x}+{\frac{3}{4}}\delta^{2}w_{xxx}
            +{\frac{3}{4}}\epsilon \delta^{2}(2w_{x}w_{xx}+ww_{xxx})\right]  \\
    =& D_{t}\left[ -\mathcal{Q}\left(w_{xxxx}+\epsilon \left( ww_{x}\right) _{xxx}\right) \right. \\
     &   \left. +{\frac{3}{4}\delta ^{2} \mathcal{Q} D_{x}^{2}}w_{xxxx} +{\frac{3}{4}}\epsilon \delta^{2} \mathcal{Q}D_{x}^{2} (2w_{x}w_{xx}+ww_{xxx})_{x}\right].
\end{align*}%
We note that the operator norms of $\mathcal{Q}$ and $\mathcal{Q}\delta ^{2}D_{x}^{2}$  are bounded:
\begin{equation*}
    \left\Vert \mathcal{Q}\right\Vert_{H^{s}}\leq 1\text{ \ and \ }
        \left\Vert \delta^{2}\mathcal{Q}D_{x}^{2}\right\Vert_{H^{s}}\leq \frac{4}{5}.
\end{equation*}%
The use of these bounds and uniform estimate for $D_{x}^{3}D_{t}^{2}w$   yield
\begin{equation}
    \left\Vert D_{x}^{3}D_{t}^{2}w\right\Vert_{H^{s}}\leq C \left\Vert D_{x}^{4}w_{t}\right\Vert_{H^{s}}
    \leq C  \left\Vert w_{t}\right\Vert_{H^{s+4}} . \label{five}
\end{equation}%
 As all the terms in $F$ have coefficients $\epsilon^{2}$,  $\epsilon^{2} \delta^{2}$,  $\delta^{4}$,  $\epsilon \delta^{4}$ or $\epsilon^{2} \delta^{4}$   (with  $0< \epsilon \leq \delta \leq 1$)\  we obtain the following estimate for the potential function
\begin{equation}
    \left\Vert F(t) \right\Vert _{H^{s}}\leq C \left( \epsilon^{2}+\delta ^{4}\right)   \left( \left\Vert w\right\Vert_{H^{s+5}}+\left\Vert w_{t}\right\Vert _{H^{s+4}}\right).  \label{saa}
\end{equation}%
Using Corollary \ref{cor2.1} with $k=5$, we obtain:
\begin{lemma}\label{lem3.1}
    Let $w_{0} \in H^{s+6}\left( \mathbb{R}\right)$, $~s>1/2$. Then, there is some $C>0$ so that the  family of solutions $w^{\epsilon,  \delta}$ to  the CH equation  (\ref{ch}) with initial value \ $w(x, 0) =w_{0}(x)$, satisfy%
\begin{equation*}
    w_{tt}-w_{xx}-\delta ^{2}w_{xxtt}-\epsilon (w^{2})_{xx}=F_{x}
\end{equation*}%
with
\begin{equation*}
    \left\Vert F\left( t\right) \right\Vert _{H^{s}}\leq C\left( \epsilon ^{2}+\delta ^{4}\right) ,
\end{equation*}%
for all $0<\epsilon \leq \delta \leq 1$ and $t\in \lbrack 0,\frac{T}{\epsilon}]$.
\end{lemma}

\section{Justification of the Camassa-Holm Approximation }\label{sec4}

In this section we prove Theorem \ref{theo4.2} given below. We have the well-posedness result for the IB equation (\ref{ib})  in a  general setting \cite{molinet,duruk2}:
\begin{theorem}
    Let $u_{0},u_{1}\in H^{s}\left ( \mathbb{R}\right)$, $~s>1/2$. Then  for any pair of parameters $\epsilon$ and $\delta$, there is some $T^{\epsilon, \delta}>0$ so that the Cauchy problem for the IB equation (\ref{ib}) with initial values $u(x, 0) =u_{0}(x)$,  $u_{t}(x, 0) =u_{1}(x)$ has a unique solution $u \in \ C^{2}\left( [0,T^{\epsilon ,\delta }],H^{s}(\mathbb{R})\right) .$ \
\end{theorem}
The existence time $T^{\epsilon ,\delta }$ above may depend on $\epsilon$ and $\delta$ and it may be chosen arbitrarily large as long as $T^{\epsilon ,\delta }<T_{\max }^{\epsilon ,\delta }$ where $T_{\max }^{\epsilon ,\delta }$ is the maximal time. Furthermore, it was shown in \cite{duruk2} that the existence time, if it is finite,  is determined by the $L^{\infty }$ blow-up condition
\begin{equation*}
    \lim_{t\rightarrow T_{\max }^{\epsilon ,\delta }} \sup\left\Vert u\left(t\right) \right\Vert _{L^{\infty }}=\infty .
\end{equation*}%

We now consider the solutions $w$ of the  CH equation with initial data $w(x,0)=w_{0}$. Then we take $w_{0}(x)$ and $w_{t}(x,0)$ as the initial conditions for the IB equation (\ref{ib}), that is,
\begin{equation*}
    u(x,0) =w_{0}(x),\text{ \ } u_{t}(x,0) =w_{t}(x,0).
\end{equation*}%
Let $u$ be the corresponding solutions of the Cauchy problem defined for the IB equation (\ref{ib}) with these initial conditions. Since $w_{0} \in H^{s+6}(\mathbb{R})$, clearly $u(x,0), u_{t}(x, 0) \in H^{s}(\mathbb{R})$. Recalling from Corollary \ref{cor2.1} that the guaranteed existence time for  $w$ is $T/\epsilon$, without loss of generality we will take $T^{\epsilon ,\delta }\leq T/\epsilon$.

In the course of our proof of  Theorem \ref{theo4.2}, we will use certain commutator estimates. We recall that the commutator is defined as $[K, L]=KL-LK$. We refer the reader to \cite{lannes} (see Proposition B.8) for the following result.
\begin{proposition}\label{prop4.2}
    \noindent\ Let $q_{0}>1/2$, $~s\geq 0$ and let  $\sigma $ be a Fourier multiplier of order $s$.
    \begin{enumerate}
    \item If \  $0\leq s\leq q_{0}+1$ and $w\in H^{q_{0}+1}$ then, for all $g\in H^{s-1}$, one has
        \begin{equation*}
            \Vert [ \sigma (D_{x}), w]g\Vert_{L^{2}}\leq C\Vert w_{x}\Vert_{H^{q_{0}}}\Vert g\Vert_{H^{s-1}},
    \end{equation*}
    \item If $-q_{0}< r\leq q_{0}+1-s$ and $w\in H^{q_{0}+1}$ then, for all $g\in H^{r+s-1}$, one has
        \begin{equation*}
            \Vert [ \sigma (D_{x}), w]g\Vert_{H^{r}}\leq C\Vert w_{x}\Vert_{H^{q_{0}}}\Vert g\Vert_{H^{r+s-1}}.
        \end{equation*}
    \end{enumerate}
\end{proposition}
 For the reader's convenience we now restate the two estimates of the above proposition as follows.  Let  $\Lambda^{s}=\left(1-D_{x}^{2}\right)^{s/2}$ and take $w\in H^{s+1}$, $g\in H^{s-1}$ and $h\in H^{s}$. Then, for $ q_{0}=s$,  the first estimate above  yields
\begin{equation}
    \langle \lbrack \Lambda^{s},w]g,\Lambda^{s}h\rangle \leq C\Vert w\Vert_{H^{s+1}}\Vert g \Vert_{H^{s-1}}\Vert h \Vert_{H^{s}}. \label{est1}
\end{equation}%
Similarly,  for $q_{0}=s$ and $-s<r\leq 1$,   we obtain from the second estimate that
\begin{align}
    \langle \Lambda \lbrack \Lambda ^{s},w]h,\Lambda ^{s-1}g\rangle
        \leq &C\Vert \Lambda \lbrack \Lambda^{s},w]h\Vert _{L^{2}}\Vert \Lambda ^{s-1}g\Vert _{L^{2}}  \notag \\
        \leq &C\Vert \lbrack \Lambda^{s},w]h \Vert_{H^{1}}\Vert g\Vert_{H^{s-1}}  \notag \\
        \leq &C\Vert w\Vert_{H^{s+1}}\Vert h \Vert_{H^{s}}\Vert g\Vert_{H^{s-1}}.  \label{est2}
\end{align}%
We are now ready to prove the main result for the CH approximation of the IB equation (an extension of the following theorem to the nonlocal equation will be given in Section \ref{sec5} (see  Theorem  \ref{theo5.2})):
\begin{theorem}\label{theo4.2}
    Let \ $w_{0} \in H^{s+6}(\mathbb{R}),$ $s>1/2$ and suppose that $w^{\epsilon, \delta}$ is the solution of the CH  equation (\ref{ch}) with initial value $w(x,0)=w_{0}(x)$. Then, there exist $T>0$ \ and $\delta_{1}\leq 1$ such that the solution $u^{\epsilon ,\delta }$ of the Cauchy problem for the IB equation
    \begin{align*}
        &  u_{tt}-u_{xx}-\delta ^{2}u_{xxtt}-\epsilon (u^{2})_{xx} =0 \\
        &  u(x,0) =w_{0}(x) ,~~~~ u_{t}(x,0)=w_{t}^{\epsilon ,\delta }(x, 0),
    \end{align*}
    satisfies
    \begin{equation*}
        \Vert u^{\epsilon, \delta}(t)-w^{\epsilon, \delta}(t) \Vert_{H^{s}}\leq ~C\left(\epsilon^{2} +\delta^{4}\right) t
    \end{equation*}
    for all $t\in \left[ 0,\frac{T}{\epsilon }\right] $ and all $0<\epsilon \leq \delta \leq \delta_{1}$.
\end{theorem}
\begin{proof}
    We fix the parameters $\epsilon$ and  $\delta  $ such that $0<\epsilon \leq \delta \leq 1$. Let $r=u-w$. We define
    \begin{equation}
        T_{0}^{\epsilon, \delta }
            =\sup \left\{ t\leq T^{\epsilon ,\delta }:\left\Vert r(\tau) \right\Vert _{H^{s}}\leq 1 ~~\mbox{for all }~ \tau \in [0, t]\right\} . \label{patlama}
    \end{equation}%
    We note that either  $\left\Vert r \left( T_{0}^{\epsilon, \delta }\right) \right\Vert _{H^{s}}=1$ or $T_{0}^{\epsilon, \delta }=T^{\epsilon, \delta }$. Moreover, in the latter case we must have $T_{0}^{\epsilon, \delta }=T^{\epsilon, \delta }=T/\epsilon$ by the discussion above about for the maximal time $T_{\max }^{\epsilon, \delta }.$ For the rest of the proof we will drop the superscripts $\epsilon, \delta $ to simplify the notation. Henceforth, we will take $t\in [0, T_{0}^{\epsilon, \delta}]$.  Obviously, the function  $r=u-w$ satisfies the initial conditions $r(x, 0) =r_{t}(x, 0) =0$. Furthermore, it satisfies the evolution equation
    \begin{equation*}
        \left( 1-\delta ^{2}D_{x}^{2}\right) r_{tt}-r_{xx}-\epsilon \left(r^{2}+2wr\right) _{xx}=-F_{x},
    \end{equation*}%
    with the residual term  $F_{x}=w_{tt}-w_{xx}-\delta ^{2}w_{xxtt}-\epsilon (w^{2})_{xx}$  that was already estimated in (\ref{saa})).  We define a function $\rho $ so that $r=\rho_{x}$ with $\rho(x,0) =\rho_{t}(x,0) =0$. This is possible since $r$ satisfies  the initial conditions $r(x,0) =r_{t}(x, 0) =0$  (see \cite{duruk2} for details).  In what follows we will use both $\rho$ and $r$  to  further simplify the calculation.  The above equation then becomes
    \begin{equation}
        \left( 1-\delta ^{2}D_{x}^{2}\right) \rho _{tt}-r_{x}-\epsilon \left(r^{2}+2wr\right) _{x}=-F . \label{rhor}
    \end{equation}%
    Motivated by the approach in \cite{gallay}, we define the "energy" as
    \begin{eqnarray}
        E_{s}^{2}(t)
        &=& \frac{1}{2}\left( \left\Vert \rho_{t}(t) \right\Vert_{H^{s}}^{2}+\delta^{2}\left\Vert r_{t}(t) \right\Vert_{H^{s}}^{2}+\left\Vert r(t) \right\Vert_{H^{s}}^{2}\right)
        +\epsilon \left\langle \Lambda^{s}(w(t) r(t)), \Lambda^{s}r(t) \right\rangle   \nonumber \\
        && +\frac{\epsilon }{2}\left\langle \Lambda^{s}r^{2}(t), \Lambda^{s}r(t) \right\rangle.  \label{ener}
    \end{eqnarray}%
    Note that
    \begin{equation*}
        \left\vert \left\langle \Lambda^{s}(wr), \Lambda^{s}r\right\rangle\right\vert \leq C\left\Vert r(t) \right\Vert_{H^{s}}^{2}, ~~~\mbox{and}~~~
        \left\vert \left\langle \Lambda^{s}r^{2}, \Lambda^{s}r\right\rangle\right\vert \leq \left\Vert r(t) \right\Vert_{H^{s}}^{3} \leq \left\Vert r(t) \right\Vert_{H^{s}}^{2},
    \end{equation*}
    where we have used (\ref{patlama}).  Thus, for sufficiently small values of $\epsilon$, we have
    \begin{equation*}
        E_{s}^{2}\left( t\right)
            \geq \frac{1}{4}\left( \left\Vert \rho_{t}\right\Vert_{H^{s}}^{2}+\delta^{2}\left\Vert r_{t}\right\Vert_{H^{s}}^{2}
                +\left\Vert r\right\Vert_{H^{s}}^{2}\right),
    \end{equation*}%
    which shows that $E_{s}^{2}(t) $ is positive definite. The above result also shows that an estimate obtained for $E_{s}^{2}$ gives an estimate for $\left\Vert r(t)\right\Vert_{H^{s}}^{2}$. Differentiating $E_{s}^{2}(t)$ with respect to $t$ and using (\ref{rhor}) to eliminate the term $\rho_{tt}$ from the resulting equation we get
    \begin{align}
        \frac{d}{dt}E_{s}^{2}
        =&\frac{d}{dt}\left( \epsilon \left\langle \Lambda^{s}(wr),\Lambda^{s}r\right\rangle
                    +\frac{\epsilon }{2}\left\langle \Lambda^{s}r^{2},\Lambda^{s}r\right\rangle\right)
                    -\epsilon \left\langle \Lambda^{s}(r^{2}+2wr),\Lambda^{s}r_{t}\right\rangle \nonumber \\
        &            -\left\langle \Lambda^{s}F,\Lambda^{s}\rho_{t}\right\rangle  \nonumber \\
        =&\epsilon \left[ \left\langle \Lambda^{s}(w_{t}r),\Lambda^{s}r\right\rangle -\left\langle \Lambda^{s}(wr),\Lambda^{s}r_{t}\right\rangle
                    +\left\langle \Lambda^{s}r,\Lambda^{s}(wr_{t})\right\rangle + \left\langle \Lambda^{s}(rr_{t}),\Lambda^{s}r\right\rangle
                   \right. \nonumber \\
        &         \left.    -\frac{1}{2}\left\langle \Lambda^{s} r^{2}, \Lambda^{s} r_{t}\right\rangle \right]
                  -\left\langle \Lambda^{s}F,\Lambda^{s}\rho_{t}\right\rangle . \label{decom}
    \end{align}%
    The  first term in the parentheses  and the last term  are  estimated as
    \begin{align*}
         \left\langle \Lambda^{s}(w_{t}r),\Lambda^{s}r\right\rangle
            \leq &C\left\Vert r\right\Vert_{H^{s}}^{2}\leq C E_{s}^{2}  \\
        \left\langle \Lambda^{s}F, \Lambda^{s}\rho_{t}\right\rangle
            \leq &\Vert F\Vert_{H^{s}}\Vert \rho_{t}\Vert_{H^{s}} \leq C\left( \epsilon ^{2}+\delta ^{4}\right)E_{s},
    \end{align*}
    respectively, where we have used Lemma \ref{lem3.1}.     We rewrite the second and the third terms in the parentheses  in (\ref{decom}) as
    \begin{align}
        -\left\langle \Lambda^{s}(wr),\Lambda^{s}r_{t}\right\rangle+\left\langle \Lambda^{s}r,\Lambda^{s}(wr_{t})\right\rangle
        =&\int \left[ -\Lambda^{s}(wr) \Lambda^{s}r_{t}+\Lambda^{s}r\Lambda^{s}(wr_{t})\right] dx \nonumber \\
        =&-\langle \lbrack \Lambda ^{s},w]r,\Lambda ^{s}r_{t}\rangle +\langle \lbrack \Lambda ^{s},w]r_{t},\Lambda^{s}r\rangle. \label{liz}
    \end{align}%
     Furthermore,  using the commutator estimates (\ref{est1})-(\ref{est2}) we get the following estimates for the two terms in (\ref{liz}):
    \begin{align}
        \langle \lbrack \Lambda ^{s},w]r,\Lambda ^{s}r_{t}\rangle
        =&\langle \Lambda \lbrack \Lambda ^{s},w]r,\Lambda ^{s-1}r_{t}\rangle
                \leq C\Vert w\Vert _{H^{s+1}}\Vert r\Vert _{H^{s}}\Vert r_{t}\Vert _{H^{s-1}},  \label{comw2} \\
        \langle \lbrack \Lambda ^{s},w]r_{t},\Lambda ^{s}r\rangle
        \leq &C\Vert w\Vert _{H^{s+1}}\Vert r\Vert _{H^{s}}\Vert r_{t}\Vert _{H^{s-1}}.%
    \end{align}%
    We rewrite the fourth and fifth terms in  the parentheses  in (\ref{decom}) as
    \begin{align*}
        \left\langle \Lambda^{s}(rr_{t}),\Lambda^{s}r\right\rangle -& \frac{1}{2}\left\langle \Lambda^{s}r^{2},\Lambda^{s}r_{t}\right\rangle  \\
        =&\left\langle \Lambda^{s-1} \left( 1-D_{x}^{2}\right)r,\Lambda^{s-1} (rr_{t})\right\rangle
            -\frac{1}{2}\left\langle \Lambda^{s-1} \left(1-D_{x}^{2}\right) r^{2},\Lambda^{s-1}r_{t}\right\rangle \\
        =& \left\langle \Lambda^{s-1}r,\Lambda^{s-1}(rr_{t})\right\rangle
            -\frac{1}{2}\left\langle \Lambda^{s-1} (r^{2}-2r_{x}^{2}),\Lambda^{s-1}r_{t}\right\rangle \\
        & -\left( \left\langle \Lambda^{s-1}(rr_{t}),\Lambda^{s-1}r_{xx}\right\rangle
            -\left\langle \Lambda^{s-1}r_{t}, \Lambda^{s-1} (rr_{xx})\right\rangle\right).
    \end{align*}%
    Then, if we group the first two terms together and the last two terms together in the above equation, we obtain the following estimates
    \begin{align*}
        \left\vert \left\langle \Lambda^{s-1}r,\Lambda^{s-1}(rr_{t})\right\rangle
            -\frac{1}{2}\left\langle \Lambda^{s-1}(r^{2}-2r_{x}^{2}),\Lambda^{s-1}r_{t}\right\rangle\right\vert
        \leq &  C\Vert r\Vert _{H^{s-1}}^{2}\Vert r_{t}\Vert _{H^{s-1}}, \\
         \leq & C\Vert r\Vert _{H^{s}}^{2}\Vert r_{t}\Vert _{H^{s-1}},  \\
        \left\vert \left\langle \Lambda^{s-1}(rr_{t}),\Lambda^{s-1}r_{xx}\right\rangle
            -\left\langle \Lambda^{s-1}r_{t},\Lambda^{s-1}(rr_{xx})\right\rangle\right\vert
        \leq &C\Vert r\Vert _{H^{s}}\Vert r_{t}\Vert _{H^{s-1}}\Vert r_{xx}\Vert _{H^{s-2}} \\
        \leq &C\Vert r\Vert _{H^{s}}^{2}\Vert r_{t}\Vert _{H^{s-1}}.
    \end{align*}%
    Note that the second line follows from (\ref{liz}) and ( \ref{comw2}) where $w,r,r_{t}$ are replaced, respectively, by $r,r_{t},r_{xx}$ and $s$ by $s-1$.  Also, we remind that
    \begin{equation*}
        \Vert r_{t}\Vert _{H^{s-1}}=\Vert \rho _{xt}\Vert _{H^{s-1}}\leq \Vert \rho_{t}\Vert _{H^{s}}\leq C E_{s}\left( t\right)
    \end{equation*}%
    and   $\Vert r\Vert _{H^{s}}\leq 1$.  Combining all the above results we get from (\ref{decom}) that
    \begin{equation*}
        \frac{d}{dt}E_{s}^{2}(t)
            \leq C\left( \epsilon E_{s}^{2}(t) +\left( \epsilon^{2}+\delta^{4}\right)E_{s}(t) \right) .
    \end{equation*}%
    As $E_{s}(0) =0$, Gronwall's inequality yields
    \begin{align*}
        E_{s}(t) \leq & \frac{\epsilon^{2}+\delta^{4}}{\epsilon }\left( e^{C\epsilon t}-1\right)
                    \leq Ce^{CT}\left( \epsilon^{2}+\delta^{4}\right)t \\
                  \leq &   C^{\prime}\left( \epsilon^{2}+\delta^{4}\right) t\text{ \ \ \ for \ \ }t\leq T_{0}^{\epsilon, \delta}\leq {T\over \epsilon}.
    \end{align*}
    Finally recall that $T_{0}^{\epsilon, \delta}$ was determined by the condition (\ref{patlama}).   The above estimate shows that
    $\left\Vert r\left(T_{0}^{\epsilon, \delta}\right)\right\Vert_{s}\leq C^{\prime}\left(\epsilon^{2}+\delta^{4}\right) T_{0}^{\epsilon, \delta}<1$ for $\epsilon\leq \delta$ small enough.  Then $T_{0}^{\epsilon, \delta}=T^{\epsilon, \delta }$ and furthermore $T^{\epsilon, \delta }=\frac{T}{\epsilon }$, and this concludes the proof.
\end{proof}

We want to conclude with some remarks about the above proof.
\begin{remark}\label{rem4.1}
     Theorem \ref{theo4.2} shows that the approximation error  is $\mathcal{O}\left((\epsilon^{2}+\delta^{4})t\right)$ for times of order $\mathcal{O}
      ({1\over \epsilon})$. Consequently, the CH approximation provides a good approximation to the solution of the IB equation for large times.
\end{remark}
\begin{remark}\label{rem4.2}
     The key step is to use   the extra $\epsilon $ terms in the energy $E_{s}^{2}$, where we have adopted the approach in \cite{gallay}. This allows us to replace $\Vert r_{t}\Vert_{H^{s}}$ by $\Vert r_{t}\Vert_{H^{s-1}}$ hence avoiding the loss of  $\delta $ in our estimates. The proofs  in \cite{gallay}  work for integer values of $s$, whereas via commutator estimates our result holds for general  $s$.   The standard approach of taking the energy as
      \begin{equation*}
        \tilde{E}_{s}^{2}(t)
        = \frac{1}{2}\left( \left\Vert \rho_{t}(t) \right\Vert_{H^{s}}^{2}+\delta^{2}\left\Vert r_{t}(t) \right\Vert_{H^{s}}^{2}+\left\Vert r(t) \right\Vert_{H^{s}}^{2}\right)  \label{ener-stan}
    \end{equation*}%
      would give the estimate
      \begin{equation*}
       \tilde{E}_{s}(t) \leq \left(\epsilon^{2}+\delta^{4}\right){\epsilon \over \delta}\left(e^{C{\epsilon \over \delta}t}-1\right),
    \end{equation*}%
    in turn implying $\tilde{E}_{s}(t) \leq C^{\prime}\left(\epsilon^{2}+\delta^{4}\right)t$   for times $t\leq \frac{\delta }{\epsilon }T$, that is, only for relatively shorter times.
\end{remark}

\section{The Nonlocal Wave Equation}\label{sec5}

In this section we return to the nonlocal equation (\ref{nw}) and extend the analysis of the previous sections concerning the IB equation (\ref{ib}) to (\ref{nw}).  We will very briefly sketch  the main features of the nonlocal equation,  referring the reader to \cite{duruk2} for more details. In \cite{duruk2}, for the propagation of strain waves in a one-dimensional, homogeneous, nonlinearly and nonlocally elastic infinite medium the following  wave equation was proposed (here we restrict our attention to the quadratically nonlinear equation):
\begin{equation}
    U_{\tau\tau}=\beta \ast ( U+ U^{2})_{XX} \label{origin}
\end{equation}
where $U=U(X,\tau)$ is a real-valued function. Following the assumptions in \cite{duruk2}, the kernel function $\beta(X)$ is even and its  Fourier transform satisfies the ellipticity condition
\begin{equation}
    c_{1}\left( 1+\eta^{2}\right)^{-r/2}\leq \widehat{\beta}(\eta) \leq c_{2}\left( 1+\eta^{2}\right)^{-r/2}  \label{order}
\end{equation}%
for some $c_{1},c_{2}>0$ and $r\geq 2$, where $\eta$ is the Fourier variable corresponding to $X$. Then the convolution can be considered as an invertible pseudodifferential operator of order $r$.  The following result on the local well-posedness of the Cauchy problem was originally given in \cite{duruk2}:
\begin{theorem}\label{theo5.1}
    Let $r\geq 2$ and  $s>1/2$. For $U_{0}, U_{1}\in H^{s}(\mathbb{R})$, there is some $\tau^{*}>0$ such that  the Cauchy problem for (\ref{origin})  with initial values $U(X,0) =U_{0}(X)$,  $U_{\tau}(X,0) =U_{1}(X)$ has a unique solution $U\in $  $C^{2}([0,\tau^{*}],H^{s}(\mathbb{R}))$.
\end{theorem}
Moreover, as in the case of the IB equation,  the $L^{\infty }$ blow-up condition
\begin{equation*}
    \lim_{\tau\rightarrow \tau_{\max }^{-}} \sup\left\Vert U(\tau)\right\Vert_{L^{\infty }}=\infty
\end{equation*}
determines the maximal existence time if it is finite.

We  note that, under the transformation defined by
\begin{equation}
 U(X, \tau) =\epsilon u(x, t), ~~~~ x=\delta X, ~~~~t=\delta \tau,  \label{scale}
\end{equation}
(\ref{origin}) becomes (\ref{nw}) with $\beta_{\delta }(x) =\frac{1}{\delta }\beta(X)=\frac{1}{\delta }\beta ( \frac{x}{\delta })$. Recall that the functional relationship between the Fourier transforms of $\beta(X)$ and $\beta_{\delta}(x)$ is as follows: $\widehat{\beta}(\eta)=\widehat{\beta}(\delta \xi)=\widehat{\beta_{\delta}}(\xi)$ where $\xi$ is the Fourier variable corresponding to $x$.  Theorem \ref{theo5.1} applies for (\ref{nw}) with  $t\in [0, T^{\epsilon, \delta}]$.   Note that if we choose the kernel function in the form $\beta_{\delta}(x)= \frac{1}{2\delta}e^{- \left\vert x\right\vert/\delta}$ (in which $\beta (X) =\frac{1}{2}e^{-\left\vert X\right\vert }$, $~\widehat{\beta}(\eta) =\left( 1+\eta^{2}\right)^{-1}$ and  $\widehat{\beta_{\delta}}(\xi) =\left( 1+\delta^{2}\xi^{2}\right)^{-1}$ ), then (\ref{nw})  recovers the IB equation  (\ref{ib}).

Our aim is to prove that, in the long-wave limit, the unidirectional solutions of the nonlocal equation are  well approximated by the solutions of the CH equation under certain minimal conditions on $\beta$ (equivalently on $\beta_{\delta}$). From now on, we will make the following assumptions on the moments of $\beta$:
\begin{equation}
        \int \beta(X) dX=1,\text{ }\int X^{2}\beta(X) dX=2, \text{ }   \int X^{4}|\beta(X)| dX<\infty. \label{beta-con}
    \end{equation}%
\begin{proposition}\label{prop5.3}
    Suppose that $\beta$ satisfies the conditions in (\ref{beta-con}). Then there is a continuous function $m$ such that
    \begin{equation}
        \frac{1}{\widehat{\beta }(\eta) }=1+\eta^{2}+\eta^{4}m(\eta).  \label{m}
    \end{equation}%
\end{proposition}
\begin{proof}
    Since the Fourier transform of $-iX\beta(X)$ equals $\frac{d}{d\eta}\widehat{\beta}(\eta)$, (\ref{beta-con}) implies that  $\widehat{\beta }\in C^{4}$ and
    \begin{equation}
        \widehat{\beta }(0)=\int \beta(X) dX=1,\text{ }(\widehat{\beta})^{\prime\prime}(0)=-\int X^{2}\beta(X) dX=-2.
    \end{equation}%
    Then $1/\widehat{\beta}(\eta)\in C^{4}$, $1/\widehat{\beta }(0)=1$ and  $\left(1/\widehat{\beta}\right)^{\prime\prime}(0)=2$. As $\beta$ is even, the odd moments, hence the odd derivatives of $1/\widehat{\beta}(\eta)$, vanish at $\eta=0$.  Thus the function defined as
    \begin{equation*}
    m(\eta)=\frac{\frac{1}{\widehat{\beta}(\eta)}-1-\eta^{2}}{\eta^{4}}
    \end{equation*}
    for $\eta\neq 0$ can be extended continuously to $\eta=0$.
\end{proof}
\begin{remark}\label{rem5.3}
    The above assumption is not very restrictive in our setting. For instance, if   $\int \beta(X) dX=a$ and $\int X^{2}\beta(X) dX=b>0$,
    a suitable scaling will reduce it to the above case.
\end{remark}

The lower bound in (\ref{order}) shows that
\begin{equation*}
   0< \frac{1}{\widehat{\beta }(\eta) }=1+\eta^{2}+\eta^{4}m(\eta)  \leq c_{1}^{-1}(1+\eta^{2})^{r/2}.
\end{equation*}
Thus
\begin{align*}
    \eta^{4} |m(\eta)| \leq  c_{1}^{-1}(1+\eta^{2})^{r/2}+(1+\eta^{2})\leq C(1+\eta^{2})^{r/2}.
\end{align*}
Since $m(\eta)$ is continuous, this implies
\begin{align*}
    |m(\eta)| \leq  C(1+\eta^{2})^{\frac{r-4}{2}},
\end{align*}
so that  $m$ has order $r-4$. We note that under the scaling (\ref{scale}) we have
\begin{equation}
        \frac{1}{\widehat{\beta_{\delta} }(\xi) }=1+\delta^{2}\xi^{2}+\delta^{4}\xi^{4}m(\delta\xi).  \label{mdel}
\end{equation}%

We define the pseudodifferential operators
\begin{equation*}
    M U=\mathcal{F}^{-1}\left( m(\eta) \widehat{U}(\eta) \right),\text{ \ \ \ }
    M_{\delta }u =\mathcal{F}^{-1}\left( m\left( \delta \xi \right) \widehat{u}(\xi) \right).
\end{equation*}%
When $r>4$, we have
\begin{equation*}
    |m(\delta\xi)|\leq C(1+\delta^{2}\xi^{2})^{\frac{r-4}{2}} \leq C (1+\xi^{2})^{\frac{r-4}{2}},
\end{equation*}%
so that
\begin{equation*}
    \left\Vert M_{\delta }u \right\Vert_{H^{s}} \leq C \left\Vert u \right\Vert_{H^{s+r-4}}.
\end{equation*}%
On the other hand, when $r\leq 4$, we get
\begin{equation*}
    |m(\delta\xi)|\leq C(1+\delta^{2}\xi^{2})^{\frac{r-4}{2}} \leq C,
\end{equation*}%
so that
\begin{equation*}
    \left\Vert M_{\delta }u \right\Vert_{H^{s}} \leq C \left\Vert u \right\Vert_{H^{s}}.
\end{equation*}
Thus we have the uniform estimates for $M_{\delta }u$:
\begin{equation}
    \left\Vert M_{\delta }u \right\Vert_{H^{s}} \leq C \left\Vert u \right\Vert_{H^{s+\sigma-4}},    ~~~~ \sigma=\max \{r, 4\}.  \label{mdelest}
\end{equation}%
Due to   (\ref{scale}), $MU=\epsilon M_{\delta}u$. Multiplying  (\ref{origin}) by $(1-D_{X}^{2}+D_{X}^{4}M)$ and (\ref{nw}) by $(1-\delta^{2}D_{x}^{2}+\delta^{4}D_{x}^{4}M_{\delta})$ we rewrite (\ref{origin}) and (\ref{nw})  more familiar forms
\begin{equation}
    \left(1-D_{X}^{2}+D_{X}^{4}M\right) U_{\tau\tau}-U_{XX}=\left(U^{2}\right)_{XX}  \label{IB nonlocal}
\end{equation}%
and
\begin{equation}
    \left( 1-\delta ^{2}D_{x}^{2}+\delta^{4}D_{x}^{4}M_{\delta} \right) u_{tt}-u_{xx}=\epsilon \left( u^{2}\right) _{xx}, \label{del-equ}
\end{equation}%
respectively.

When we apply the formal asymptotic approach given in \cite{eee1} to (\ref{del-equ}) (in \cite{eee1} it was used to derive the CH equation from the IB equation), we again get  exactly the same result, that is, the CH equation. As remarked in \cite{eee1}, this follows from the observation that the extra term $\delta^{4}D_{x}^{4}M_{\delta}$ will only give rise to $O\left( \delta ^{4}\right) $ terms and these terms do not affect the derivation  in \cite{eee1}. The following theorem gives the convergence of the formal asymptotic expansion and shows that  the right-going solutions of (\ref{del-equ}) (and (\ref{nw})) are well approximated by the solutions of the CH  equation.
\begin{theorem}\label{theo5.2}
    Let \ $\ w_{0}\in H^{s+\sigma+2}(\mathbb{R})$, $s>1/2$, $\sigma=\max \{r, 4\}$ and suppose $w^{\epsilon ,\delta }$ is the solution of the  CH equation (\ref{ch}) with initial value $w(x,0)=w_{0}(x)$. Then, there exist $T>0$ \ and $\delta_{1}\leq 1$  such that  the solution $u^{\epsilon, \delta}$ of the Cauchy problem for  (\ref{del-equ}) (equivalently for (\ref{nw}))
    \begin{align*}
    & \left( 1-\delta ^{2}D_{x}^{2}+\delta ^{4}D_{x}^{4}M_{\delta }\right)u_{tt}-u_{xx}-\epsilon (u^{2})_{xx} =0, \\
    & u(x,0) =w_{0}\left( x\right) ,\text{ }u_{t}\left( x,0\right)=w_{t}^{\epsilon ,\delta }\left( x,0\right),
    \end{align*}%
    satisfies
    \begin{equation*}
        \Vert u^{\epsilon, \delta}(t) -w^{\epsilon, \delta}(t) \Vert_{H^{s}}\leq ~C\left(\epsilon^{2} +\delta^{4}\right) t
    \end{equation*}%
    for all $\ t\in \left[ 0,\frac{T}{\epsilon }\right] $ and all $\ 0<\epsilon \leq \delta \leq \delta_{1}$.
\end{theorem}
\begin{proof}
    The proof follows a similar pattern to that of the proof of Theorem \ref{theo4.2}. The only difference is that  (\ref{del-equ}) involves  additional term $\delta ^{4}D_{x}^{4}M_{\delta }u_{tt}$. Following closely  the scheme in the proof of Theorem \ref{theo4.2} corresponding to case of the IB equation, we now outline the proof.  First we note that plugging  the solution $w^{\epsilon, \delta}$ of the CH equation into  (\ref{del-equ}) leads to a residual term $D_{x}F^{M}$ with $F^{M}=F+\delta^{4}D_{x}^{3}M_{\delta }w_{tt}$ where $D_{x}F$ is the residue term corresponding to the IB case, given in (\ref{fff}). Going through a cancelation process  similar to the cancelations in the IB case, we get
    \begin{equation*}
        \left\Vert D_{x}^{3}M_{\delta}w_{tt}\right\Vert_{H^{s}} \leq  C \left\Vert D_{x}^{3}w_{tt}\right\Vert_{H^{s+\sigma-4}}
                \leq C  \left\Vert w_{t}\right\Vert_{H^{s+\sigma-4+4}}  = C  \left\Vert w_{t}\right\Vert_{H^{s+\sigma}},
    \end{equation*}%
    where we use the estimate  (\ref{mdelest}) for $M_{\delta}$ and (\ref{five}) for $D_{x}^{3}w_{tt}$. Since $\sigma \geq 4$, we have
    \begin{equation*}
        \left\Vert F^M(t)\right\Vert_{H^{s}} \leq  C (\epsilon^{2}+\delta^{4})
                \left( \left\Vert w\right\Vert_{H^{s+\sigma+1}}+ \left\Vert w_{t}\right\Vert_{H^{s+\sigma}}\right).
    \end{equation*}%
    Thus we take $k=\sigma +1$ in Corollary \ref{cor2.1}  to get a uniform bound on $F^{M}$. The next step is to define the energy as
    \begin{equation*}
        E_{s,M}^{2}=E_{s}^{2}+\frac{1}{2}\delta^{4}\left\langle \Lambda^{s}M_{\delta }D_{x}r_{t}(t), \Lambda^{s}D_{x}r_{t}(t)\right\rangle,
    \end{equation*}%
    where $E_{s}^{2}$ is given by (\ref{ener}). We note that the extra term in $E_{s,M}^{2}$ is not necessarily positive. Yet recalling that $r=\rho_{x}$ and collecting  the $\rho_{t}$ and $r_{t}$ terms in $E_{s,M}^{2}$ we have:
    \begin{align*}
        \left\Vert \rho_{t}\right\Vert_{H^{s}}^{2}+\delta^{2}\left\Vert r_{t}\right\Vert _{H^{s}}^{2}
            -\delta ^{4}\left\langle\Lambda^{s} D_{x}^{2}M_{\delta}r_{t},\Lambda^{s}r_{t}\right\rangle
            =&\left\langle \Lambda^{s}\left(1-\delta^{2}D_{x}^{2}+\delta^{4}D_{x}^{4}M_{\delta }\right) \rho _{t},\Lambda^{s}\rho _{t}\right\rangle  \\
            =&\int \frac{\left(1+\xi^{2}\right)^{s}}{\widehat{\beta} (\delta\xi) }\left\vert \widehat{\rho_{t}}(\xi) \right\vert^{2}d\xi  \\
            \geq & c_{2}^{-1}\int (1+\xi ^{2})^{s}( 1+\delta^{2}\xi^{2})^{r/2}\left\vert \widehat{\rho _{t}}(\xi) \right\vert^{2}d\xi  \\
            \geq &c_{2}^{-1}\int (1+\xi^{2})^{s} (1+\delta^{2}\xi^{2}) \left\vert \widehat{\rho_{t}}(\xi) \right\vert^{2}d\xi  \\
            =& c_{2}^{-1}\left( \left\Vert \rho_{t}\right\Vert_{H^{s}}^{2}+\delta^{2}\left\Vert \rho_{xt}\right\Vert _{H^{s}}^{2}\right)  \\
            =&c_{2}^{-1}\left( \left\Vert \rho _{t}\right\Vert _{H^{s}}^{2}+\delta^{2}\left\Vert r_{t}\right\Vert _{H^{s}}^{2}\right) .
    \end{align*}%
Hence $E_{s,M}^{2}\geq C E_{s}^{2}$. It is straightforward to compute the time derivative of $E_{s, M}^{2}$ since as the extra term
vanishes due to (\ref{IB nonlocal}) and  we are left with the same right-hand side as in the previous section and hence with the same
conclusion.
\end{proof}
\begin{remark}\label{rem5.4}
    We conclude from Theorem \ref{theo5.2} that the comments made in Remark \ref{rem4.1} on the precision of the CH approximation to the IB equation are also valid for the nonlocal equation.
\end{remark}

\section{The BBM and KdV Approximations}\label{sec6}

In this section we consider the BBM equation and the KdV equation which characterize the particular cases  of the CH equation and we show
how the results of the previous sections can be used to obtain the results for these two equations. The analysis is similar in spirit to that of Sections \ref{sec3} and \ref{sec4}, we therefore give only the main steps in the proofs.

\subsection{The BBM Approximation}

When we neglect  terms of order $\epsilon \delta^{2}$ in the CH equation (\ref{ch}), we get the BBM equation
\begin{equation}
    w_{t}+w_{x}+\epsilon ww_{x}-{\frac{3}{4}}\delta^{2}w_{xxx}-{\frac{5}{4}}\delta^{2}w_{xxt}=0,  \label{bbm}
\end{equation}
which is a well-known model for unidirectional propagation of long  waves in shallow water \cite{bbm}. It should be noted that, in order to write this equation in a more standard form, the term $w_{xxx}$  can be eliminated by means of the coordinate transformation given in Section \ref{sec1}. Obviously, the BBM equation (\ref{bbm}) is a special case of (\ref{cons-a})  with $\kappa_{1}=1$, $\kappa_{2}=\kappa_{3}=0$, $\kappa_{4}=-\frac{3}{4} $, $\kappa_{5}=- \frac{5}{4}$ and $\kappa_{6}=\kappa_{7}=0$. Then, for the BBM equation, Corollary \ref{cor2.1} takes the following form:
\begin{corollary}\label{cor6.2}
    Let $w_{0} \in H^{s+k+1}\left( \mathbb{R}\right)$,  $s>1/2$,  $k\geq 1$. Then, there exist $T>0$, $C>0$  and a unique family of solutions
    \begin{equation*}
        w^{\epsilon ,\delta }\in C\left( [0,\frac{T}{\epsilon }],H^{s+k}(\mathbb{R})\right)
            \cap C^{1}\left( [0,\frac{T}{\epsilon }],H^{s+k-1}(\mathbb{R})\right)
    \end{equation*}%
    to the BBM equation (\ref{bbm}) with initial value $w(x,0)=w_{0}(x)$, satisfying
    \begin{equation*}
        \left\Vert w^{\epsilon ,\delta }\left( t\right) \right\Vert _{H^{s+k}}
            +\left\Vert w_{t}^{\epsilon ,\delta }\left( t\right) \right\Vert_{H^{s+k-1}}\leq C,
    \end{equation*}%
    for all $0< \delta \leq 1$, $\epsilon \leq \delta $ and $t\in \lbrack 0,\frac{T}{\epsilon }\rbrack$.
\end{corollary}
As we did in Section \ref{sec3}, we  plug the solution $w$ of  the Cauchy problem of the BBM equation  into the IB equation. Then the residual term $f$ is given by (\ref{residual0}) but now $w$ represents a solution of the BBM equation. Making use of the approach in Section \ref{sec3}, we obtain $f$ corresponding to the case of the BBM approximation in the form $f=F_{x}$ with
\begin{align*}
    F
    =&\epsilon^{2} \left({w^{3}\over 3}\right)_{x}-{1\over 4}\epsilon\delta^{2}\left(6ww_{xxt}+2w_{x}w_{xt}+w_{t}w_{xx}-9w_{x}w_{xx}\right) \\
    & +\frac{1}{16}\delta^{4}D_{x}^{3}\left(5w_{tt}-12w_{xt}-9w_{xx}\right).
\end{align*}%
 Thus we have the BBM version of Lemma \ref{lem3.1}, namely the uniform estimate
\begin{equation*}
    \left\Vert F(t) \right\Vert _{H^{s}}\leq C\left( \epsilon ^{2}+\delta ^{4}\right).
\end{equation*}%
The rest of the proof holds and we obtain the BBM version of Theorem  \ref{theo4.2}:
\begin{theorem}\label{theo6.1}
    Let  $w_{0} \in H^{s+6}(\mathbb{R})$, $s>1/2$ and suppose $w^{\epsilon ,\delta }$ is the solution of the BBM equation (\ref{bbm})  with initial value $w(x,0)=w_{0}(x)$. Then, there exist $T>0$  and $\delta_{1}\leq 1$  such that  the solution $ u^{\epsilon ,\delta }$ of the Cauchy problem for the IB equation
    \begin{align*}
       & u_{tt}-u_{xx}-\delta ^{2}u_{xxtt}-\epsilon (u^{2})_{xx} =0 \\
       & u(x,0) =w_{0}(x), \text{ }u_{t}(x, 0) =w_{t}^{\epsilon, \delta}(x,0),
    \end{align*}%
    satisfies
    \begin{equation*}
        \Vert u^{\epsilon, \delta}(t) -w^{\epsilon, \delta}(t) \Vert _{H^{s}}\leq ~C\left(\epsilon^{2} +\delta ^{4}\right) t
    \end{equation*}%
    for all $\ t\in \left[ 0,\frac{T}{\epsilon }\right] $ and all $\ 0<\epsilon \leq \delta \leq \delta_{1}$.
\end{theorem}
Following the arguments in Section \ref{sec5}, we may extend Theorem  \ref{theo6.1} to the general class of nonlocal wave equations, namely
\begin{theorem}\label{theo6.2}
    Let  $w_{0} \in H^{s+\sigma+2}(\mathbb{R})$, $s>1/2$, $\sigma=\max \{r, 4\}$ and suppose $w^{\epsilon ,\delta}$ is the solution of the BBM equation (\ref{bbm})  with initial value $w(x,0)=w_{0}(x)$. Then, there exist $T>0$  and $\delta_{1}\leq 1$  such that  the solution $ u^{\epsilon ,\delta }$ of the Cauchy problem for the nonlocal equation
    \begin{align*}
       & u_{tt}=\beta_{\delta }\ast ( u+\epsilon u^{2})_{xx} \\
       & u(x,0) =w_{0}(x), \text{ }u_{t}(x, 0) =w_{t}^{\epsilon, \delta}(x,0),
    \end{align*}%
    satisfies
    \begin{equation*}
        \Vert u^{\epsilon, \delta}(t) -w^{\epsilon, \delta}(t) \Vert _{H^{s}}\leq ~C\left(\epsilon^{2} +\delta ^{4}\right) t
    \end{equation*}%
    for all $\ t\in \left[ 0,\frac{T}{\epsilon }\right] $ and all $\ 0<\epsilon \leq \delta \leq \delta_{1}$.
\end{theorem}

\subsection{The KdV Approximation}

The KdV equation \cite{korteweg}
\begin{equation}
    w_{t}+w_{x}+\epsilon ww_{x}+{\frac{\delta ^{2}}{2}}w_{xxx}=0  \label{kdv}
\end{equation}
is also a well-known model for unidirectional propagation of long  waves in shallow water and it has the same order of accuracy as the BBM equation.  In fact, the KdV equation (\ref{kdv}) is a special case of (\ref{cons-a})  with $\kappa_{1}=1$, $\kappa_{4}=1/2$, $\kappa_{2}=\kappa_{3}=\kappa_{5}=\kappa_{6}=\kappa_{7}=0$. However, Proposition \ref{prop2.1} will not apply to the KdV equation because the condition $\kappa_{5}<0$ is not satisfied. Instead we refer to the following theorem proved by Alazman et al in \cite{alazman}:
\begin{theorem}\label{theo6.3} (Theorem A2 in \cite{alazman})
    Let $s\geq 1$ \ be an integer. Then for every $K>0$, there exists $C>0$ such that the following is true. Suppose $q_{0}\in H^{s}$ with $\left\Vert q_{0}\right\Vert_{H^{s}}\leq K$, and let $q$ be the solution of the KdV equation
    \begin{equation}
        q_{t}+q_{x}+\frac{3}{2}\bar{\epsilon} qq_{x}+\frac{1}{6}\bar{\epsilon} q_{xxx}=0 \label{alazkdv}
    \end{equation}
    with initial data $q(x,0)=q_{0}(x)$. Then for all $\bar{\epsilon} \in (0,1]$ and all $ t\geq 0$,
    \begin{equation*}
        \left\Vert q(t)\right\Vert_{H^{s}}\leq C.
    \end{equation*}
    Further, for every integer $l$ such that $1\leq 3l\leq s$, it is the case that
    \begin{equation*}
        \left\Vert D_{t}^{l}q(t)\right\Vert_{H^{s-3l}}\leq C.
    \end{equation*}
\end{theorem}
It is easy to see that the substitution
\begin{equation}
        w={\frac{9}{2}} {\frac{\delta^{2}}{\epsilon}}q, ~~~~~\delta^{2}=\frac{\bar{\epsilon}}{3}
\end{equation}
transforms (\ref{kdv}) into (\ref{alazkdv}). Suppose $c_{1}\leq {\frac{\delta^{2}}{\epsilon}} \leq c_{2}$ with positive constants $c_{1}$ and $c_{2}$. Then we have
\begin{equation}
    \left\Vert w(t)\right\Vert_{H^{s}}
        =  {\frac{9}{2}} {\frac{\delta^{2}}{\epsilon}} \left\Vert q(t)\right\Vert_{H^{s}}
        \leq {\frac{9}{2}} c_{2}\left\Vert q(t)\right\Vert_{H^{s}}  \label{w-estimate}
\end{equation}
and
\begin{equation}
    \left\Vert q_{0}\right\Vert_{H^{s}}
        =  {\frac{2}{9}} {\frac{\epsilon}{\delta^{2}}} \left\Vert w_{0}\right\Vert_{H^{s}}
        \leq {\frac{2}{9c_{1}}} \left\Vert w_{0}\right\Vert_{H^{s}}.  \label{w0-estimate}
\end{equation}
We thus reach the following corollary:
\begin{corollary}\label{cor6.3}
    Let $s+k\geq 1$  be an integer. Suppose $w_{0}\in H^{s+k}$  and let $w^{\epsilon, \delta}$ be the solution of the KdV equation (\ref{kdv})     with initial data $w(x,0)=w_{0}(x)$. Then there is some $C$ such that for all  $\delta^{2} \in (0,\frac{1}{3}]$ and all $\epsilon \in \left[\frac{\delta^{2}}{c_{2}}, \frac{\delta^{2}}{c_{1}}\right]$ with positive constants $c_{1}$ and $c_{2}$ and  all  $ t\geq 0$,
    \begin{equation*}
        \left\Vert w^{\epsilon, \delta}(t)\right\Vert_{H^{s+k}}+ \left\Vert w^{\epsilon, \delta}_{t}(t)\right\Vert_{H^{s+k-3}}  \leq C.
    \end{equation*}
\end{corollary}

We  next plug the solution $w^{\epsilon,\delta}$ of   the KdV equation (\ref{kdv})  into the IB equation. Again, omitting the indices $\epsilon,\delta  $, the residual term $f$ is given by (\ref{residual0}).  Following the steps in Section \ref{sec3}, we obtain $f$ corresponding to the case of the KdV approximation in the form $f=F_{x}$ with
\begin{equation*}
    F=D_{x}\left\{\frac{1}{3}\epsilon^{2}w^{3}+\frac{1}{4}\epsilon\delta^{2}\left[-3(w_{x})^{2}+4(ww_{x})_{t}\right]
        +\frac{1}{4}\delta^{4}(-w_{xxxx}+2w_{xxxt})\right\}.
\end{equation*}
As there are at most five derivatives of $w$ and four derivatives of $w_{t}$ in $F$, we will choose $k=7$ in the corollary to get the KdV version of  Lemma \ref{lem3.1}, namely the estimate:
    \begin{equation*}
        \left\Vert F(t) \right\Vert _{H^{s}}\leq C \epsilon^{2}
    \end{equation*}%
for the residual term.

Although the above results hold for all times, to follow the approach in the previous sections we fix some $T>0$ and restrict ourselves to the time interval $[0, {T\over \epsilon}]$. As in the previous cases, the residual estimate leads to the following theorem:
\begin{theorem}\label{theo6.4}
    Let  $w_{0} \in H^{s+7}(\mathbb{R}),$ $s\geq 1$ an integer and suppose $w^{\epsilon, \delta}$ is the solution of the KdV equation (\ref{kdv})  with  initial value $w(x,0)=w_{0}(x)$. Then, for any $T>0$ and  $0<c_{1}<c_{2}$   there exist   $\delta_{1}^{2}\leq {1\over 3}$ and $C>0$ such that  the solution $ u^{\epsilon, \delta}$ of the Cauchy problem for the IB equation
    \begin{align*}
       & u_{tt}-u_{xx}-\delta^{2} u_{xxtt}-\epsilon (u^{2})_{xx} =0 \\
       & u(x,0) =w_{0}(x), \text{ }u_{t}(x, 0) =w_{t}^{\epsilon, \delta}(x,0),
    \end{align*}%
    satisfies
    \begin{equation*}
        \Vert u^{\epsilon, \delta}(t) -w^{\epsilon, \delta}(t) \Vert _{H^{s}}\leq ~C \epsilon^{2}  t
    \end{equation*}%
    for all $\ t\in \left[ 0,\frac{T}{\epsilon }\right] $ and  all $\delta \in (0, \delta_{1}]$, $\epsilon \in \left[\frac{\delta^{2}}{c_{2}}, \frac{\delta^{2}}{c_{1}}\right]$.
\end{theorem}
The result in Theorem \ref{theo6.4}, namely the rigorous justification of the KdV approximation of the IB equation, was already proved by Schneider \cite{schneider}. The discussion in Section \ref{sec5} allows us to prove a similar theorem for the general class of nonlocal wave equations.  Again we have to estimate the term $D_{x}^{3}M_{\delta }w_{tt}$ in the residue $F^{M}$. We get
\begin{equation*}
        \left\Vert D_{x}^{3}M_{\delta}w_{tt}\right\Vert_{H^{s}} \leq  \left\Vert w_{tt}\right\Vert_{H^{s+3+\sigma-4}}
                \leq C  \left\Vert w\right\Vert_{H^{s+3+\sigma-4+6}}  = C  \left\Vert w\right\Vert_{H^{s+\sigma+5}},
\end{equation*}%
which requires taking $k=\sigma+5$ in Corollary \ref{cor6.3}. Hence we get:
\begin{theorem}\label{theo6.5}
    Let  $w_{0} \in H^{s+\sigma+5}(\mathbb{R}),$ $s>1/2$, $s+\sigma$ an integer, $\sigma=\max\{r, 4\}$ and suppose $w^{\epsilon, \delta}$ is the solution of the KdV equation (\ref{kdv})  with  initial value $w(x,0)=w_{0}(x)$. Then, for any $T>0$ and  $0<c_{1}<c_{2}$   there exist   $\delta_{1}^{2}\leq {1\over 3}$ and $C>0$  such that  the solution $ u^{\epsilon, \delta}$ of the Cauchy problem for the nonlocal equation
    \begin{align*}
       & u_{tt}=\beta_{\delta }\ast (u+\epsilon u^{2})_{xx} \\
       & u(x,0) =w_{0}(x), \text{ }u_{t}(x, 0) =w_{t}^{\epsilon, \delta}(x,0),
    \end{align*}%
    satisfies
    \begin{equation*}
        \Vert u^{\epsilon, \delta}(t) -w^{\epsilon, \delta}(t) \Vert _{H^{s}}\leq ~C \epsilon^{2}  t
    \end{equation*}%
    for all $\ t\in \left[ 0,\frac{T}{\epsilon}\right] $ and all $\delta \in (0, \delta_{1}]$,  $\epsilon \in \left[\frac{\delta^{2}}{c_{2}}, \frac{\delta^{2}}{c_{1}}\right]$.
\end{theorem}
We finally note that in the KdV case $T$ can be chosen arbitrarily large while in the CH or the BBM cases $T$ is determined by the equation.

\section*{Acknowledgments} Part of this research was done while the third author was visiting the Institute of Mathematics at the Technische Universit\"at Berlin. The third author wants to thank Etienne Emmrich and his group for their warm hospitality.

\medskip
% The data information below will be filled by AIMS editorial staff
%Received xxxx 20xx; revised xxxx 20xx.
\medskip

\end{document}